\UseRawInputEncoding
\documentclass[10pt, oneside]{amsart}
\usepackage{geometry}           
\makeindex

\usepackage{amsmath, amsfonts, amssymb, amstext, amscd, graphicx, hyperref, url, enumerate}
\allowdisplaybreaks
\newtheorem{definition}{Definition}

\newtheorem{theorem}{Theorem}
\newtheorem{lemma}{Lemma}
\newtheorem{corollary}{Corollary}

\newtheorem{conjecture}{Conjecture}
\theoremstyle{remark}

\newcommand {\bb} [1] {\mathbb{#1}}

\newcounter{example}[section]
\newenvironment{example}[1][]{\refstepcounter{example}\par\medskip
\noindent \textit{Example~\theexample. #1} \rmfamily}{\medskip}

\input xy
\xyoption{all}

\begin{document}

\begin{center}\textsc{\Large Monotone Catenary Degree in Numerical Monoids\footnote{Written in 2015 during the PURE Math undergraduate summer research program and funded by NSF award \#1045147.}}\\
    D. Gonzalez Cedre, C. Wright, J. Zomback
\end{center}

\section{Abstract}

Recent investigations on the catenary degrees of numerical monoids have demonstrated that this invariant is a powerful tool in understanding the factorization theory of this class of monoids.  Although useful, the catenary degree is largely not sensitive to the lengths of factorizations of an element.  In this paper, we study the monotone catenary degree of numerical monoids, which is a variant of catenary degree that requires chains run through factorization lengths monotonically.  In general, the monotone catenary is greater than or equal to the catenary degree.  We begin by providing an important class of monoids (arithmetical numerical monoids) for which monotone catenary degree is equal to the catenary degree.  Conversely, we provide several classes of embedding dimension $3$ numerical monoids where monotone catenary degree is strictly greater.  We conclude by showing that this difference can grow arbitrarily large.

\section{Introduction}

A \textbf{numerical monoid} is a cofinite additive submonoid of $\mathbb{N}$. It is known that for every monoid of this type there exists a unique minimal generating set, so for a given numerical monoid $M$, we can write 
$$M=\langle n_1,\ldots,n_k\rangle=\{a_1n_1+\cdots+a_kn_k|a_1,\ldots,a_k\in\mathbb{N}\}.$$
In this case, we say that the monoid $M$ is of \textbf{embedding dimension} $k$. One interesting property of numerical monoids is that they do not in general have unique factorization. For example, the element $25$ in the monoid $\langle 5,7,9\rangle$ has precisely two factorizations, namely $5(5)$ and $7+2(9)$. We refer to each factorization in vector form, so these two factorizations correspond to the vectors $(5,0,0)$ and $(0,1,2)$. Since the study of numerical monoids has begun, several combinatorial invariants have been introduced to help  gain understanding of the structures. These invariants include the \textbf{catenary degree}, the \textbf{monotone catenary degree}, \textbf{adjacent catenary degree}, and \textbf{equivalent catenary degree}. Previously, there has been much research conducted on the standard catenary degree, but until now not much was known about the various types of catenary degrees in numerical monoids. It is clear from the definition of monotone catenary degree that it will always be greater than or equal to the regular catenary degree, but in general it is not well understood when this inequality is strict. The purpose of this paper is to explore the relationship between the monotone and regular catenary degrees, utilizing the equivalent and adjacent variations to aid our understanding. \\
\indent The results that follow describe the monotone catenary degree in various families of monoids, beginning with monoids generated by arithmetic sequences, proceeding to those generated by generalized arithmetic sequences, and ending with other notable families. We take this approach to illustrate how the monotone catenary degree is not as well-behaved as the standard catenary degree. In the case of monoids generated by arithmetic sequences the two values coincide, but many restrictions must be applied for this to hold in monoids generated by generalized arithmetic sequences. Further, we show that there are families of monoids for which the difference between their monotone and regular catenary degrees can be arbitrarily large.

\section{Definitions and Preliminaries}

Let $M=\langle n_1,\ldots,n_k\rangle$ be a numerical monoid and take some element $m\in M$. We will define the \textbf{set of factorizations} of the element $m$ to be
$$\mathcal{Z}(m)=\{(a_1,\ldots,a_k)\in\mathbb{N}^{k}|a_1n_1+\cdots+a_kn_k=m\}.$$
Now, given some $z\in\mathcal{Z}(m)$, we will say the \textbf{length} of $z=(z_1,\ldots,z_k)$, denoted $|z|$, is the sum of its components, $|z|=z_1+\cdots+z_k$. Additionally, we refer to the \textbf{length set} of $m$ as 
$$\mathcal{L}(m)=\{n\in\mathbb{N}|n=|z|\text{ for some }z\in\mathcal{Z}(m)\}.$$
Finally, we say that $L(m)=\max\{\mathcal{L}(m)\}$ and $l(m)=\min\{\mathcal{L}(m)\}$. The definition of the length set gives rise to another notable set, namely the \textbf{delta set} of an element, $\Delta(m)$, defined as follows:
$$\Delta(m)=\{b-a|a,b\in\mathcal{L}(m),\text{ }a<b,\text{ and }[a,b]\cap\mathcal{L}(m)=\{a,b\}\}.$$
This delta set is not particularly important for the bulk of our discussion, but it will prove useful in our treatment of monoids generated by arithmetic sequences. Together, all of the above concepts are of great importance to us in this paper, as the nature of the monotone catenary degree forces us to pay close attention to the lengths of specific factorizations. Furthermore, we can use our definition of length to construct a metric on the set $\mathcal{Z}(m)$. To do this, though, we must first understand the notion of a greatest common divisor of two factorizations. 
\begin{definition}
    Given $z,z'\in\mathcal{Z}(m)$, we will define the \textbf{greatest common divisor} of $z$ and $z'$ to be
    $$\gcd\{z,z'\}=(\min\{z_1,z'_1\},\min\{z_2,z'_2\},\ldots,\min\{z_k,z'_k\}).$$
\end{definition}
Intuitively, we can say that the greatest common divisor of $z$ and $z'$ is the $k$-tuple signifying all of the generators that the two factorizations have in common. The length of a greatest common divisor is defined similarly to that of factorizations; given $\gcd\{z,z'\}=(g_1,\ldots,g_k)$, we say $|\gcd\{z,z'\}|=g_1+\cdots+g_k$. This notion gives rise to our metric.
\begin{definition}
    For factorizations $z,z'\in\mathcal{Z}(m)$, we refer to the \textbf{distance} between $z$ and $z'$ as
    $$d(z,z')=\max\{|z|,|z'|\}-|\gcd\{z,z'\}|.$$
\end{definition}
\indent This distance allows us to define the notion of an $\mathbf{N}$\textbf{-chain} between two factorizations $z,z'\in\mathcal{Z}(m)$. 
\begin{definition} 
    An $N$-chain joining $z$ and $z'$ is a finite sequence of factorizations $z=z_0,z_1,\ldots z_n=z'$ such that $d(z_i,z_{i+1})\leq N$ for all $i\geq 0$.
\end{definition}
We use the concept of $N$-chains to define the catenary degree. 
\begin{definition}
    The \textbf{catenary degree} of an element $m\in M$, denoted $c(m)$, is the smallest natural number $N$ such that every two factorizations $z,z'\in\mathcal{Z}(m)$ are connected by an $N$-chain.
\end{definition}
This particular variation of the catenary degree will be referred to as the regular catenary degree. From here, we turn our attention to the various other types of catenary degree, which we will study at length throughout this paper. These variants are all similar to the regular catenary degree, but each either observes specific subsets of $\mathcal Z(m)$ or considers different types of $N$-chains. First, we discuss the equivalent and adjacent catenary degrees.
\begin{definition}
    The \textbf{equivalent catenary degree} of an element $m\in M$, denoted $c_{eq}(m)$, is the smallest natural number $N$ such that any two factorizations of the same length are connected by an $N$-chain with the restriction that $|z_i|=|z_{i+1}|$ for all $0\leq i\leq n-1$.
\end{definition}
To define the adjacent catenary degree, we need to define the notion of adjacency within the length set. \begin{definition}
    Take two lengths $a,b\in\mathcal{L}(m)$ and suppose without loss of generality that $a<b$. The lengths  $a$ and $b$ are deemed \textbf{adjacent} if $[a,b]\cap\mathcal{L}(m)=\{a,b\}$.
\end{definition}
Additionally, let
$$\mathcal{Z}_a(m)=\{z\in\mathcal{Z}(m)||z|=a\}.$$
\begin{definition}
    The \textbf{adjacent catenary degree} of an element $m\in M$, denoted $c_{adj}(m)$, is the smallest natural number $N$ such that $d(\mathcal{Z}_a,\mathcal{Z}_b)\leq N$ for all adjacent $a,b\in\mathcal{L}(m)$.
\end{definition}
Finally, we define the monotone catenary degree of an element.
\begin{definition}
    The \textbf{monotone catenary degree} of an element $m\in M$ is the minimum natural number $N$ such that any two factorizations $z,z'\in\mathcal{Z}(m)$ are connected by an $N$-chain subject to the constraint that $|z_i|\leq|z_{i+1}|$ or $|z_i|\geq|z_{i+1}|$ for all $0\leq i\leq n-1$.
\end{definition}
The above definitions all pertain to the various kinds of catenary degrees of elements, but our main goal in this paper is to discuss the set-wise catenary degrees. We define the catenary degree of the monoid $M$, denoted $c(M)$, as 
$$c(M)=\max\{c(m)|m\in M\}.$$
The set-wise monotone, equivalent, and adjacent catenary degrees are defined analogously. \\
\indent The primary purpose of this paper is to discuss the monotone catenary degree of numerical monoids, but we include analysis of the equivalent and adjacent catenary degrees in order to assist us in our understanding. The following theorem illustrates why these catenary degrees are useful to our discussion.

\begin{theorem}
    Let $M$ be a numerical monoid. Then for all $m\in M$, $c_{mon}(m) = \max\{c_{eq}(m), c_{adj}(m)\}$.
\end{theorem}

\begin{proof}
    Let $M$ be a Numerical Monoid and $m \in M$. Assume for the sake of contradiction that
    $c_{mon}(m) \not = \max\{c_{eq}(m), c_{adj}(m)\}.$
    \begin{enumerate}
        \item[Case 1:] $c_{mon}(m) < c_{eq}(m)$\\
            \indent We know that there exists a monotone $c_{mon}(m)$-chain connecting any two factorizations of equal length
            because chains formed by factorizations of the same length are monotone. But, by definition, $c_{eq}(m)$ is the least
            $n \in \mathbb N$ such that there exists an $N$-chain between any two factorizations of equal length, contradiction our assumption that $c_{mon}(m) < c_{eq}(m)$.

        \item[Case 2:] $c_{mon}(m) < c_{adj}(m)$\\
            \indent Recall that $c_{mon}(m)$ is the least $n \in \mathbb N$ such that there exists a monotone $n$-chain between any two
            factorizations. By definition, $c_{adj}(m)$ is the least $n \in \mathbb N$ such that, given $l_1$ and $l_2$ adjacent lengths in
            $\mathcal L(m)$, the minimum distance between factorizations of length $l_1$ and factorizations of length $l_2$ is $\leq n$.
            From this definition it follows that there exist adjacent $l_1, l_2 \in \mathcal L(m)$ such that, for all
            $z, z' \in \mathcal Z(m)$, if $|z| = l_1$ and $|z'| = l_2$, then
            $$d(z, z') \geq c_{adj}(m).$$
            So, the minimum distance between any two factorizations of lengths $l_1$ and $l_2$ is $c_{adj}(m)$. Let $z$ and $z'$ be two
            factorizations of lengths $l_1$ and $l_2$ respectively. By definition, we have a monotone $c_{mon}(m)$-chain connecting
            $z$ to $z'$. At some point, this monotone chain must connect a factorization of length $l_1$ with a factorization of length
            $l_2$ through a distance $\geq c_{adj}(m)$. However, every distance in the $c_{mon}(m)$-chain must be
            $\leq c_{mon}(m)$, contradicting the assumption that $c_{mon}(m) < c_{adj}(m)$.

        \item[Case 3:]$c_{mon}(m) > c_{eq}(m)$ and $c_{mon} > c_{adj}(m)$\\
            \indent Let $z, z' \in \mathcal Z(m)$. Let $f = (f_1,\ldots, f_k) \in \mathcal Z(m)^k$ be a
            monotone chain such that $z = f_1$ and $f_k = z'$ and, for all $i \in \{1,\ldots, k - 1\}$, either
            $|f_i| = |f_{i + 1}|$ or $|f_i|$ and $|f_{i + 1}|$ are adjacent in $\mathcal L(m)$. Let $i \in \{1,\ldots,k - 1\}$.
            If $|f_i| = |f_{i + 1}|$, there is a $c_{eq}(m)$-chain between $f_i$ and $f_{i + 1}$, so we can choose $f_i$ and
            $f_{i + 1}$ such that $d(f_i, f_{i + 1}) \leq c_{eq}(m)$. If $|f_i|$ and $|f_{i + 1}|$ are adjacent in
            $\mathcal L(m)$ then $f_i$ and $f_{i + 1}$ can be chosen so that $d(f_i, f_{i + 1}) \leq c_{adj}(m)$. This makes
            $f$ a $\max\{c_{eq}(m), c_{mon}(m)\}$-chain. This process can be applied to any two $z$ and $z'$ in $\mathcal Z(m)$,
            producing a $\max\{c_{eq}(m), c_{mon}(m)\}$-chain between them. However, $c_{mon}(m)$ is the least $n \in \mathbb N$
            such that there is a monotone $n$-chain between any two factorizations in $\mathcal Z(m)$.
    \end{enumerate}

    Therefore, $c_{mon}(m) = \max\{c_{eq}(m), c_{adj}(m)\}.$
\end{proof}

Note that as a result of this, we see that $c_{mon}(M)=\max\{c_{eq}(M),c_{adj}(M)\}$. Both of these facts prove vital to us for the remainder of this paper.

\section{Arithmetic Monoids}

We consider first the class of monoids generated by arithmetic sequences, or rather monoids whose generators are of the form $a,a+d,\ldots,a+kd$ for positive integers $a,k$, and $d$. Henceforth, we will refer to these monoids as \textbf{arithmetic monoids}. The primary reason for working with arithmetic monoids is that they behave in a very well-understood fashion. Much research has been done in the area of arithmetic numerical monoids and, as such, we have many tools to work with when approaching the monotone catenary degree. For instance, it is known from \cite{bowles} that in an arithmetic monoid $M$, $\Delta(m)=d$ for all $m\in M$. Additionally, Theorem 3.8 in \cite{omidali} explicitly tells us the catenary degree of an arithmetic monoid.

\begin{theorem}
    In an arithmetic monoid $M$ such that $M=\langle a,a+d,\ldots,a+kd\rangle$, $c(M)=\left\lceil\frac{a}{k}\right\rceil+d$.
\end{theorem}

Since the monotone catenary degree can be effectively studied using the equivalent and adjacent catenary degrees, the following result about arithmetic monoids informs us that our main concern in this section is the adjacent catenary degree. 

\begin{theorem}
    For all elements $m$ of an arithmetic monoid $M$ in arbitrary embedding dimension, $c_{eq}(m)=2$.
\end{theorem}
\begin{proof}
    Let $m\in M$ and take any two factorizations $z$ and $z'$ of $m$ such that $|z|=|z'|$. Lemma 4.5 from \cite{corrales} tells us that we can construct a 2-chain of factorizations $z=z_0,z_1,\ldots,z_n=z'$ such that $|z_i|=|z|=|z'|$ for $0\leq i\leq n$. Thus our theorem is proven. 
\end{proof}
From here, we proceed to view individual factorizations, with the goal of finding specific factorizations from which we can alway step to an adjacent factorization length, producing a distance less than or equal to the catenary degree of the monoid. This will allow us to conclude that the adjacent catenary degree of the element is equal to the regular catenary degree and subsequently that the monotone catenary degree of an element will be equal to the regular catenary degree of the same element.\\ 
\indent For the following result, it is important to construct the notion of a length-preserving move between factorizations in an arithmetic monoid. In \cite{corrales}, Lemma 4.2 tells us that in any such monoid $M$, if some factorization $z=(z_0,z_1,\ldots,z_k)$ of an element $m\in M$ is of length $|z|$ and there exist $z_i$ and $z_j$ such that $z_i\geq 1$ and $z_j \geq 1$ and $1\leq i<j\leq k-1$, then there exists another factorization $z'$, also of length $|z|$, such that  
$$z'=(\ldots,z_{i-1}+1,z_i-1,\ldots,z_j-1,z_{j+1}+1,\ldots),$$
where the ellipses represent entries unchanged from those of $z$. Furthermore, if we can apply the same move in the scenario when $i=j$, so long as we now introduce the constraint $i=j\geq 2$. For these moves, we will say that a move is centered at $i$ and $j$, or centered at $i$ in the case where $i=j$. We will refer to these moves in the proof of the following lemma.

\begin{lemma}
    Let $M$ an arithmetic monoid and $m\in M$. Now, let $l\in \mathcal{L}(m)$. Then there exists a factorization $z$ such that $|z|=l$ and $z$ is of the form
    $$z=(\alpha,0,\ldots,0,1,0,\ldots,0,\beta)\text{ or }z=(\alpha,0,\ldots,0,\beta).$$
    For some natural numbers $\alpha$ and $\beta$.
\end{lemma}
\begin{proof}
    We prove by induction on $k$, which will here be taken as the embedding dimension. As the base case, we observe when $k=3$. In this case, take some element $m$ of an arithmetic monoid  $M$ and a factorization $z\in \mathcal{Z}(m)$ of length $l$, say $z=(z_1,z_2,z_3)$. Then apply our length preserving move for embedding dimension three, namely the move centered at position 1, $\lfloor\frac{z_2}{2}\rfloor$ times, yielding the factorization $z'$ where
    $$z'=\left(z_1+\left\lfloor\frac{z_2}{2}\right\rfloor,z_2-2\left\lfloor\frac{z_2}{2}\right\rfloor,z_3+\left\lfloor\frac{z_2}{2}\right\rfloor\right).$$
    It is clear that the second coordinate of $z_2$ will be either 1 or 0, and thus we have a factorization of length $l$ of the desired form.\\
    \indent Now, we assume that our hypothesis holds for embedding dimension $n$. To be precise, we suppose that for any monoid $M=\langle a,a+d,\ldots,a+(k-1)d\rangle$, given an element $m\in M$ and a factorization $z$ of $m$ of length $l$, we can generate a factorization $z'$ of $m$, also of length $l$, such that $z'$ is of the form $z'=(\alpha,0,\ldots,0,1,0,\ldots,0,\beta)$ or $z'=(\alpha,0,\ldots,0,\beta)$. Take some arithmetic monoid $M_0$ of embedding dimension $k+1$, say $M_0=\langle a,a+d,\ldots,a+kd\rangle$. Then it must be that there exists a monoid $M$ of embedding dimension $k$ such that $M=\langle a,a+d,\ldots,a+(k-1)d\rangle$. Now, take $m_0\in M_0$ and some factorization $y$ of $m_0$ such that $y=(y_0,\ldots,y_k)$. Note that $n=m_0-y_k(a+kd)\in M$, so $z=(y_0,\ldots,y_{k-1})\in\mathcal{Z}(n)$ and by our inductive hypothesis we can construct $z'\in\mathcal{Z}(n)$ such that $|z'|=|z|$ and either $z'=(\alpha^{(1)},0,\ldots,0,1,0,\ldots,0,\beta^{(1)})$ or $z'=(\alpha^{(1)},0,\ldots,0\beta^{(1)})$. Thus we can construct a factorization of $m_0$, $y^{(1)}$, where
    $$y^{(1)}=(\alpha^{(1)},0,\ldots,0,1,0,\ldots,0,\beta^{(1)},y_k)\text{ or }y^{(1)}=(\alpha^{(1)},0,\ldots,0,\beta^{(1)},y_k).$$
    Now, we apply our length-preserving moves centered at $k-1$ until we cannot any more, yielding one of the four following factorizations
    $$\left(\alpha^{(1)},\ldots,1,\ldots,\left\lfloor\frac{\beta^{(1)}}{2}\right\rfloor,0,y_k+\left\lfloor\frac{\beta^{(1)}}{2}\right\rfloor\right)$$
    $$\left(\alpha^{(1)},\ldots,1,\ldots,\left\lfloor\frac{\beta^{(1)}}{2}\right\rfloor,1,y_k+\left\lfloor\frac{\beta^{(1)}}{2}\right\rfloor\right)$$
    $$\left(\alpha^{(1)},\ldots,\left\lfloor\frac{\beta^{(1)}}{2}\right\rfloor,0,y_k+\left\lfloor\frac{\beta^{(1)}}{2}\right\rfloor\right)$$
    $$\left(\alpha^{(1)},\ldots,\left\lfloor\frac{\beta^{(1)}}{2}\right\rfloor,1,y_k+\left\lfloor\frac{\beta^{(1)}}{2}\right\rfloor\right).$$
    At this point, we can again consider the first $k$ entries of this factorization as a factorization of the element $m_0-(y_k+\lfloor\frac{\beta}{2}\rfloor)(a+kd)$ and put it into one of the two desired forms, yielding a new factorization of $m_0$ of the form
    $$y^{(2)}=\left(\alpha^{(2)},0,\ldots,0,1,0,\ldots,0,\beta^{(2)},y_k+\left\lfloor\frac{\beta^{(1)}}{2}\right\rfloor\right)\text{ or }y''=\left(\alpha^{(2)},0,\ldots,0,\beta^{(2)},y_k+\left\lfloor\frac{\beta^{(1)}}{2}\right\rfloor\right).$$
    Since $|y^{(2)}|=|y^{(1)}|=|y|$, it must be that $\alpha^{(1)}+\beta^{(1)}\geq \alpha^{(2)}+\beta^{(2)}$. Hence in all cases when the inequality is strict, we can repeat this process until we achieve a factorization of one of the following forms
    $$y^{(d)}=\left(\alpha^{(d)},0,\ldots,0,1,0,\ldots,0,1,y_k+\left\lfloor\frac{\beta^{(1)}}{2}\right\rfloor+\cdots+\left\lfloor\frac{\beta^{(d-1)}}{2}\right\rfloor\right)$$
    $$y^{(d)}=\left(\alpha^{(d)},0,\ldots,0,1,0,\ldots,0,y_k+\left\lfloor\frac{\beta^{(1)}}{2}\right\rfloor+\cdots+\left\lfloor\frac{\beta^{(d-1)}}{2}\right\rfloor\right)$$
    $$y^{(d)}=\left(\alpha^{(d)},0,\ldots,0,1,0,y_k+\left\lfloor\frac{\beta^{(1)}}{2}\right\rfloor+\cdots+\left\lfloor\frac{\beta^{(d-1)}}{2}\right\rfloor\right)$$
    $$y^{(d)}=\left(\alpha^{(d)},0,\ldots,0,y_k+\left\lfloor\frac{\beta^{(1)}}{2}\right\rfloor+\cdots+\left\lfloor\frac{\beta^{(d-1)}}{2}\right\rfloor\right)$$
    Note that all but the first of these options is already of the form described in the lemma. In the first option, assuming that the center 1 is in the $i$th position, we can simply apply a length-preserving move centered at $i$ and $k-1$ to put this factorization into the appropriate form.\\
    \indent We must additionally address the possibility that $\alpha^{(x)}+\beta^{(x)}=\alpha^{(x+1)}+\beta^{(x+1)}$ for some $x$, since in this case it is not clear that our process will terminate. In this scenario, it must be that $\beta^{(x)}=1\text{ or }0$ since $\alpha^{(x+1)}+\beta^{(x+1)}=\alpha^{(x)}+\beta^{(x)}+\lfloor\frac{\beta^{(x)}}{2}\rfloor$. If $\beta^{(x)}=0$, then we are done, but if $\beta^{(x)}=1$ then we must have passed over the factorization
    $$\left(\alpha^{(x)},0,\ldots,0,1,0,\ldots,0,1,y_k+\left\lfloor\frac{\beta^{(1)}}{2}\right\rfloor+\cdots+\left\lfloor\frac{\beta^{(x-1)}}{2}\right\rfloor\right),$$
    in which case we can simply apply a length-preserving move centered at $i$ and $k-1$ in order to yield a factorization of the desired form.

\end{proof}

We will call a factorization of this form \textit{optimized}. To see a concrete instance of the process described in the proof above, we consider the following example:

\noindent \textit{Example:} Observe the arithmetic monoid $M=\langle 11,15,19,23,27\rangle$. We have one factorization $(0,1,0,3,1)$ of $111$ in $M$ and wish to generate an optimized factorization of the same length. To do this, we apply a move centered at 3 to produce the factorization $(0,1,1,1,2)$ and subsequently apply a move centered at 1 and 3, yielding $(1,0,1,0,3)$. This is an optimized factorization with $\alpha=1$, $\beta=1$, and $i=2$. 

\begin{lemma}
    Let $M$ be an arithmetic monoid $M=\langle a,a+d,\ldots,a+kd\rangle$. Then for some natural number $m\in\mathbb{N}$, $m\in M$ if and only if there exist natural numbers $c_1,c_2$ with the constraint $kc_1\geq c_2$ such that $m=c_1n+c_2s$.
\end{lemma}
\begin{proof}
    This result follows directly from the statements of Lemma 7 and Lemma 10 from \cite{chapman}.
\end{proof}

Observe that, in an optimized factorization $z$ where $z=(\alpha,0,\ldots,0,1,0,\ldots,0,\beta)$ or $z=(\alpha,0,\ldots,0,\beta)$ and, in the first case, the 1 is understood to be in the $i$th coordinate, the constraint in the previous lemma, $kc_1\geq c_2$, is equivalent to the constraints $k(\alpha+\beta+1)\geq i+\beta k$ and $k(\alpha+\beta)\geq \beta k$, respectively. Additionally, from this lemma we can conclude that such a factorization is of minimum length in $\mathcal{L}(m)$ if reducing the length of the factorization would violate our lemma, or rather if
$$k(\alpha+\beta+1-d)<i+\beta k+a\text{ or }k(\alpha+\beta-d)<\beta k+a.$$
This leads us to the conclusion that an optimized factorization is \textit{not} of minimum length whenever 
\newline$k(\alpha+\beta+1-d)\geq i+\beta k+a$ or $k(\alpha+\beta-d)\geq i+\beta k+a$.

\noindent \textit{Example:} Again observe factorizations of 111 in the monoid $M=\langle 11,15,19,23,27\rangle$, where $a=11$, $d=4$, and $k=4$. We take an optimized factorization of length 9 and use the above relation to show that this factorization is not of minimum length. Note that $(8,0,0,1,0)\in\mathcal{Z}(111)$ is an optimized factorization with $\alpha=8$, $\beta=0$, and $i=3$. We see that $k(\alpha+\beta+1-d)=4(8+1-4)=20$ and $i+\beta k+a=3+0(4)+11=14$. Since $20>14$, it must be that $(8,0,0,1,0)$ is not of minimum length, hence we must have at least one factorization of length 5.

\begin{lemma}
    Let $M$ an arithmetic monoid and $m\in M$. If $z\in\mathcal{Z}(m)$ is an optimized factorization and $|z|\neq l(m)$, then $\alpha\geq c(M)-1=\left\lceil\frac{a}{k}\right\rceil+d-1$.
\end{lemma}
\begin{proof}
    We consider our proof in two cases, namely when the factorization $z$ takes each of the two forms possible for an optimized factorization.
    \begin{enumerate}

        \item[Case 1:] First, we let $z=(\alpha,0,\ldots,0,1,0,\ldots,0,\beta)$, where our 1 is understood to be in the $i$th coordinate, hence it is the coefficient for $a+id$ within the expansion of $m$. Suppose that $z$ is not of minimum length, or rather $|z|\neq l(m)$. Recall that in this scenario, a factorization of non-minimal length must satisfy the relation $k(\alpha+\beta+1-d)\geq i+\beta k+a$.
            In manipulating this inequality, we see
            \begin{align*}
                k(\alpha+\beta+1-d)&\geq i+\beta k+a\\
                \alpha+\beta+1&\geq \frac{1}{k}(i+a+\beta k)+d\\
                \alpha&\geq \frac{i+a}{k}+d-1
            \end{align*}
            Now, since $\alpha$ is a natural number, we can conclude that $\alpha\geq \left\lceil\frac{i+a}{k}\right\rceil+d-1\geq \left\lceil\frac{a}{k}\right\rceil +d-1$.

        \item[Case 2:] Second, we consider when $z=(\alpha,0,\ldots,0,\beta)$. Again, suppose that $|z|\neq l(m)$ and recall our relation from before, now $k(\alpha+\beta-d)\geq \beta k+a$.
            Manipulating this, we have
            \begin{align*}
                k(\alpha+\beta-d)&\geq \beta k+a\\
                \alpha+\beta-d&\geq\frac{1}{k}(\beta k+a)\\
                \alpha&\geq\frac{a}{k}+d
            \end{align*}
            Again invoking the fact that $\alpha$ is a natural number, we see that this is equivalent to \newline $\alpha\geq\left\lceil\frac{a}{k}\right\rceil+d>\left\lceil\frac{a}{k}\right\rceil+d-1$.
    \end{enumerate}
    Hence our lemma is proven.
\end{proof}

\begin{lemma}
    In an arithmetic monoid $M=\langle a,a+d,\ldots, a+kd \rangle$, we can use the element $b\in M$ of the form $b=a\left(\left\lceil\frac{a}{k}\right\rceil+d \right)$ to construct a move $B$ of factorizations such that, for an arbitrary factorization $z$ of an element $m\in M$, $|Bz|=|z|-d$ and $d(z,Bz)=c(M)=\left\lceil\frac{a}{k}\right\rceil+d$.
\end{lemma}

\begin{proof}
    First, we let $M=\langle a,a+d,\ldots,a+kd\rangle$ be an arithmetic monoid. We observe the element $y$, defined as above. Lemma 3.7 in \cite{omidali} tells us that
    $$\min\{n\in \bb{N}\cup \{0\}| na\in \langle a+d,\ldots, a+kd\rangle\}=\left\lceil\frac{a}{k}\right\rceil+d,$$
    hence $y\in\langle a+d,a+2d,\ldots, a+kd\rangle$. Additionally, within the same paper Proposition 3.4, applied to our scenario, states that for an element $x=q(a+d)+id\in\langle a+d,\ldots,a+kd\rangle$, 
    $\max L(x)=q+\left\lceil\frac{i}{k}\right\rceil(h-1)=q$. Here, the second equality comes from the fact that we are in an arithmetic monoid, hence $h=1$.
    To apply this proposition for our purposes, we can rewrite $y$ as follows
    \begin{align*}
        y&=a\left(\left\lceil\frac{a}{k}\right\rceil+d\right)\\
        &=\left\lceil\frac{a}{k}\right\rceil a+ad\\
        &=\left\lceil\frac{a}{k}\right\rceil (a+d)+\left(a-\left\lceil\frac{a}{k}\right\rceil\right)d
    \end{align*}
    thus, here $q=\lceil\tfrac{a}{k}\rceil$ and $i=\left(a-\lceil\tfrac{a}{k}\rceil\right)$. Now, we see that $\max\{\mathcal{L}(y)\}=\left\lceil\frac{a}{k}\right\rceil$. This tells us that there exists some factorization of $y$ of length $\lceil\tfrac{a}{k}\rceil$. We will call it $z_0$. Now, we use this factorization $z_0$ and the factorization $z_1=(c(M),0,\ldots,0)$ of $y$ in $M$ to construct our move $B$; namely, given $m\in M$ and $z\in\mathcal{Z}(m)$, $Bz=z-z_1+z_0$. Now, notice that $|Bz|=|z|-d$, since $|Bz|=|z-z_1+z_0|=|z|-c(M)+\lceil\frac{a}{k}\rceil=|z|-d$. All that remains is to compute the distance yielded by an application of this move. Upon inspection, we see that $\gcd(z,Bz)=(z_0-c(M),z_1,\ldots,z_k)$, so $|\gcd(z,Bz)|=|z|-c(M)$ and
    \begin{align*}
        d(z,Bz)&=|z-\gcd(z,Bz)|\\
        &=|z|-(|z|-c(M))\\
        &=c(M)
    \end{align*}
\end{proof}

\noindent \textit{Example:} We observe the monoid $M=\langle 11,15,19,23,27\rangle$ and use the element described in the previous lemma to construct a length-changing move with associated distance the catenary degree. Since $c(M)=7$ and $a=11$, our element is $b=77$ and we have $\mathcal{Z}(b)=\{(7,0,0,0,0),(0,0,0,1,2)\}$ and so we have a move $B$ such that for some factorization $z$, $Bz=z-(7,0,0,0,0)+(0,0,0,1,2)$. Now, take the element 111 in $M$ and consider the factorization $z=(7,1,1,0,0)$. Applying $B$, we see $Bz=(7,1,1,0,0)-(7,0,0,0,0)+(0,0,0,1,2)=(0,1,1,1,2)$. Now we have $|z|=9$, $|Bz|=|(0,1,1,1,2)|=5$ and $d(z,Bz)=7=c(M)$.

\begin{theorem}
    Let $M$ be an arithmetic monoid $M=\langle a,a+d,\ldots,a+kd\rangle$.Then for all elements $m\in M$ such that $c(m)\notin \{0,2\}$, 
    $$c_{adj}(m)=c(m)=c(M)$$
\end{theorem}
\begin{proof}
    Take $m\in M$, where $M=\langle a,a+d,\ldots,a+kd\rangle$ and additionally, let $z\in\mathcal{Z}(m)$ such that $z$ is optimized and $|z|\neq l(m)$. First, we consider when $z$ takes the form $z=(\alpha,0,\ldots,0,\beta)$. Then, by Lemma 3, Case 2 of this paper, we have that $\alpha>\left\lceil\frac{a}{k}\right\rceil+d-1$. Now we can use our move $B$ from Lemma 4 to step down to some factorization $z'$ of $m$ where $|z'|=|z|-d$. Additionally, this yields a distance $d(z,z')=c(M)=c(m)$. Next, we consider when $z=(\alpha,0,\ldots,0,1,0,\ldots,0,\beta)$. In this case, if $\alpha\geq\left\lceil\frac{a}{k}\right\rceil+d$, then we can again use our move $B$, stepping down to a factorization $z'$ of length $|z|-d$, again giving a distance $d(z,z')=c(M)=c(m)$. Finally, we must consider the case where $z$ again takes this form, but $\alpha=\left\lceil\frac{a}{k}\right\rceil+d-1$, and so
    $$z=\left(\left\lceil\frac{a}{k}\right\rceil+d-1,0,\ldots,0,1,0,\ldots,0,\beta\right)$$
    where the 1 is understood to be in the $i$th position. Since $|z|\neq l(m)$, there exists at least one factorization of $m$ of length $\alpha+\beta+1-d$ and ore remarks from earlier tells us that $z$ satisfies the inequality
    $$k(\alpha+\beta+1-d)=k(\left\lceil\frac{a}{k}\right\rceil+\beta)\geq i+\beta k+a.$$
    Thus, expanding $m$, we can see that
    \begin{align*}
        m&=(\alpha+\beta+1)a+(i+\beta k)d\\
        &=(\alpha+\beta+1-d)a+(i+\beta k+a)d\\
        &=\left(\left\lceil\frac{a}{k}\right\rceil+\beta\right)a+(i+\beta k+a)d.
    \end{align*}
    Now we take some element $n\in M$ where $n=m-\beta(a+kd)$, so $n$ has expansion
    \begin{align*}
        n&=(\alpha+1)a+(i)d\\
        &=(\alpha+1-d)a+(i+a)d\\
        &\geq \left(\left\lceil\frac{a}{k}\right\rceil+d-1+1-d\right)a+(i+a)d\\
        &=\left\lceil\frac{a}{k}\right\rceil a+(i+a)d
    \end{align*}
    Additionally, since $z$ satisfies the inequality
    $$k\left(\left\lceil\frac{a}{k}\right\rceil+\beta\right)\geq i+\beta k+a,$$
    we immediately see that $k\left\lceil\frac{a}{k}\right\rceil\geq i+a$.
    This inequality tells us that we have a factorization of $n$ of length $\left\lceil\frac{a}{k}\right\rceil=\alpha+1-d$. Let $y$ denote a factorization of $n$ of length $\alpha+1-d$. Then we have a factorization $z'$ of $m$ of length $\alpha+\beta+1-d$ such that we can write $z'=(0,\ldots,0,\beta)+y$. Now, we compute the distance from $z$ to $z'$. 
    $$d(z,z')=|z|-|\gcd(z,z')|=\left\lceil\frac{a}{k}\right\rceil+\beta+d-|\gcd(z,z')|.$$
    All that remains is to assess the value of $|\gcd(z,z')|$. Since $z=\left(\left\lceil\frac{a}{k}\right\rceil+d-a,0,\ldots,0,1,0,\ldots,0,\beta\right)$ and $z'=(0,\ldots,0,\beta)+y$, it must be the case that $|\gcd(z,z')|\geq\beta$. Thus, returning to our computation, we achieve
    \begin{align*}
        d(z,z')&=\left\lceil\frac{a}{k}\right\rceil+\beta+d-|\gcd(z,z')|\\
        &\leq \left\lceil\frac{a}{k}\right\rceil+d\\
        &=c(M)
    \end{align*}
    Now, we have seen that for an arbitrary non-minimal factorization length $l$, we can always step from some factorization $z$ of length $l$ to a factorization $z'$ of length $l-d$, producing a distance of at most $c(m)=c(M)$. This tells us that, for our arbitrary element $m\in M$ with $c(m)\notin\{0,2\}$,
    $$c_{adj}(m)=c(m)=c(M),$$
    and our theorem is proven.
\end{proof}

\begin{theorem}
    Let $M$ be an arithmetic monoid generated by $\{a,a+d,\ldots,a+kd\}$. Now, take an element $m\in M$. If $c(m)=2$ or $c(m)=0$, then $c(m)=c_{mon}(m)$.
\end{theorem}
\begin{proof}
    If $c(m)=0$, Theorem 4.1 from \cite{corrales} tells us that $|\mathcal{Z}(m)|=1$, so it must be that $c_{mon}(m)=0=c(m)$. If, instead, $c(m)=2$, then this same theorem tells us that $|\mathcal{L}(m)|=1$, hence every chain of factorizations is monotone, since all factorizations are of the same length. In this case, $c_{mon}(m)=c_{eq}(m)=2=c(m)$.
\end{proof}

\begin{theorem}
    Given an arithmetic monoid $M=\langle a,a+d,\ldots,a+kd\rangle$, for all $m\in M$,
    $$c_{mon}(m)=c(m).$$
    As a result of this, we also have that $c_{mon}(M)=c(M)$.
\end{theorem}
\begin{proof}
    First, consider when $c(m)\in\{0,2\}$. In this case, Theorem 4 tells us that $c_{mon}(m)=c(m)$. Now, consider when $c(m)\notin\{0,2\}$, hence $c(m)>2$. Here, Theorem 3 tells us that $c_{adj}(m)=c(m)=c(M)$. Because $c_{mon}(m)=\max\{c_{eq}(m),c_{adj}(m)\}$ and $c_{eq}(m)=2$, it must be that $c_{mon}(m)=c_{adj}(m)=c(m)$. Finally, since we now have that $c_{mon}(m)=c(m)$ for all elements $m\in M$, it must also be that $c(M)=c_{mon}(M)$.
\end{proof}


\section{Generalized Arithmetic Monoids}

We have now seen that, for arithmetic monoids, the regular and monotone catenary degrees are always equal. The situation for
generalized arithmetic monoids is significantly more nuanced. In order to simplify our considerations, we restrict our discussion to
embedding dimension 3 monoids of the form $M = \left \langle a,ah + d,ah + 2d \right \rangle$ where
$h \in \mathbb{N} \setminus \{0, 1\}$.

In general, we know \cite{bowles} that the delta set of any numerical monoid has smallest element $\gcd(h - 1, d)$ and largest element
$\left \lceil \frac{a}{2} \right \rceil (h - 1) + d$. In regular arithmetic monoids, we recall that these two values are equal to the
step size of the arithmetic sequence generating the monoid. However, in generalized arithmetic monoids, these two values are never
equal, leading to unpredictable delta sets and complicated distances between adjacent factorizations. This important difference
makes the calculation of bounds for the adjacent catenary degree more involved than in regular arithmetic sequences, leading to
a wider range of values that the adjacent catenary degree can take, which affects what values the monotone catenary can take.

Another important difference is that the equivalent catenary degree in arithmetic monoids is \textit{always} 2. The equivalent
catenary degree in generalized arithmetic monoids, however, is illustrated by the following theorem:

\begin{theorem}
    For a generalized arithmetic monoid M in embedding dimension 3, $$\displaystyle c_{eq}(M) = \frac{ah + 2d - a}{\gcd(h - 1, d)}.$$
\end{theorem}

\begin{proof}
    Observe that, given a factorization $z = (z_1, z_2, z_3) \in \mathcal Z(g)$, all factorizations of length $|z|$
    have the following form (for $k \in \mathbb Z$):
    $$\left (z_1 + k \frac{d}{\gcd(h - 1,d)},
    z_2 - k \frac{ah + 2d - a}{\gcd(h - 1,d)},
    z_3 + k \frac{ah + d - a}{\gcd(h - 1d)} \right ).$$
    This transformation produces a distance of
    $$\displaystyle k\frac{a(h - 1) + d}{\gcd(h - 1,d)}$$
    between the two factorizations. Therefore, the minimum $n$-chain that can be constructed between factorizations of the same
    length is the one produced by $k \in \{1, -1\}$ since whenever we have factorizations of the same length related by the above
    transformation for a value of $k \geq 1$, we must have factorizations ``between" them corresponding to values of $k$ less than
    that one (similarly for values of $k \leq -1$ and factorizations with $k$\--values greater), so that a
    $\frac{ah + 2d - a}{\gcd(h - 1,d)}$-chain can be constructed connecting them. Therefore, the minimum
    $n \in \mathbb N$ such that there is an $n$-chain between any two factorizations of the same length is
    $$\frac{ah + 2d - a}{\gcd(h - 1,d)}.$$
\end{proof}

The dependence of the equivalent catenary degree on $\gcd(h - 1, d)$ means that when $\gcd(h - 1, d) = 1$ the equivalent catenary
degree can more easily overcome the adjacent catenary degree (since it will take on a ``large" value) compared to when
$\gcd(h - 1, d) > 1$, leading the monotone catenary degree to be maxed out by either the equivalent or the adjacent catenary degree
depending on the value of $\gcd(h - 1, d)$.

One should note that the above length\--preserving transformation is a special case of the one that acts on the indeces of
$z$ in the following way:
$$ \left ( z_1 - k \frac{n_3 - n_1}{n_2 - n_1} - k - l
,z_2 + k \frac{n_3 - n_1}{n_2 - n_1}
,z_3 + k\right ),$$
which describes the co\--ordinates of all factorizations with length $|z| + l$ (where $k \in \mathbb Z$ is such that the coordinates
are non\--negative).
In order to ensure $c(M) < c_{mon}(M)$ we must meet two conditions. We must first show that, for all $m \in M$ that, between any two
factorizations $z, z' \in \mathcal Z(m)$ such that $|z| = |z'|$, there exists a third factorization $y \in \mathcal Z(m)$ adjacent in
length to $z$ and $z'$ such that $d(z, y) < c_{eq}(m)$ and $d(z', y) < c_{eq}(m)$. If we additionally have that
$c_{eq}(M) > c{adj}(M)$, then we will be able to conclude that $c(M) < c_{mon}(M)$.

The following pair of conjectures characterize the regular and monotone catenary degrees of generalized arithmetic monoids
depending on the different values assumed by $\gcd(h - 1, d)$ and $c_{eq}(M)$:

\begin{conjecture}
    If $\gcd(h - 1,d) > 1$, then $c_{mon}(M) = c(M)$.
\end{conjecture}

\begin{conjecture}
    If $\gcd(h - 1,d) = 1$, there are several cases:\\
    If $h < d$, then we have that $c(M) < c_{mon}(M)$.\\
    If $h \geq d$ and $c(M) < c_{eq}(M)$, then $c(M) < c_{mon}(M)$.\\
    If $h \geq d$ and $c(M) = c_{eq}(M)$, then $c(M) = c_{mon}(M)$.
\end{conjecture}

Furthering our support for the above conjectures, we believe that in the case that $\gcd(h - 1 ,d) = 1$, the first condition required
for $c(M) < c_{mon}(M)$ is true; this is illustrated by the following conjecture:

\begin{conjecture}
    Assume $\gcd(h - 1,d) = 1$ and that there exist $z, z' \in \mathcal Z(m)$ such that $z \neq z'$ and $|z| = |z'|$ and
    $z_1 > z'_1$. Then, there exists $f \in \mathcal Z(m)$ such that
    $$|f| = |z| + 1 \text{ and } d(f, z) \leq d(z, z') \text{ and } d(f, z') \leq d(z, z').$$
    Further,
    \begin{align*}
        \text{if } h \leq d, \text{ then } 
        \displaystyle &f = \left (z_1 + \left \lceil \frac{a}{2} \right \rceil
        ,z_2 - \frac{a - \frac{ah + d - \left \lceil \frac{a}{2} \right \rceil (ah + d - a)}{d}}{ah + d - a}
        ,z_3 - \frac{ah + d - \left \lceil \frac{a}{2} \right \rceil (ah + d - a)}{d} \right ),
        \\
        \text{if } h > d, \text{ then }
        \displaystyle &f = \left (z_1 + \left \lfloor \frac{a}{2} \right \rfloor
        ,z_2 - \frac{a - \frac{ah + d - \left \lfloor \frac{a}{2} \right \rfloor (ah + d - a)}{d}}{ah + d - a}
        ,A z_3 - \frac{ah + d - \left \lfloor \frac{a}{2} \right \rfloor (ah + d - a)}{d} \right ).
    \end{align*}
\end{conjecture}

\section {Conditions for Strict Inequality}
Recall that for all numerical monoids $M$, $c_{mon}(M) \geq c(M)$. Likewise, for any element $m \in M$, $c_{mon}(m) \geq c(m)$. We have seen that in arithmetic monoids, the regular and monotone catenary degrees are equal both element-wise and set-wise. However, in other classes of monoids, we observe that $c_{mon}(M) > c(M)$. In fact, we can place certain restrictions on our monoid to ensure that this strict inequality holds.  
\begin{theorem}
    Given the following two conditions, $c_{mon}(m)>c(m)$ for any element $m$ in any monoid $M$.
    \begin{itemize}
        \item $c_{eq}(m)>c_{adj}(m)$ 
        \item If $m$ has two factorizations $z_1$ and $z_2$ of length $l$, there exists a factorization $z_3$ of length $q \neq l$ such that $d(z_1,z_3)<c_{eq}(m)$ and $d(z_2,z_3)<c_{eq}(m)$. 
    \end{itemize}
\end{theorem}

\begin{proof}
    Let $M$ be a numerical monoid with $c_{eq}(m)>c_{adj}(m)$ for some $m \in M$, and guarantee that if $m$ has two factorizations $z_1$ and $z_2$ of length $l$, there exists a factorization $z_3$ of length $q \neq l$ such that $d(z_1,z_3)<c_{eq}(m)$ and $d(z_2,z_3)<c_{eq}(m)$. We claim that we can connect any two factorizations of $m$ with monotone $N$-chains such that $N<c_{eq}(m)$. Since $c_{eq}(m)>c_{adj}(m)$, we can connect a factorization of each length to one of an adjacent length with distance less than $c_{eq}(m)$. Additionally, our other constraint ensures that we can get from two factorizations of length $l$ to one of a lower length with distance less than $c_{eq}(m)$.  \\
    \indent Thus if our two conditions our met, we can connect any two factorizations of $m$ with monotone $N$-chains such that $N<c_{eq}(m)$. Then $c(m)<c_{eq}(m)$ and $c_{mon}(m)=\max\{c_{eq}(m),c_{adj}(m)\}=c_{eq}(m)$. Hence, $c_{mon}(m)>c(m)$.
\end{proof}

\begin{corollary}
    Given that the following two conditions hold, $c_{mon}(M)>c(M)$ for any monoid $M$.
    \begin{itemize}
        \item $c_{eq}(M)>c_{adj}(M)$ 
    \item If $m \in M$ has two factorizations $z_1$ and $z_2$ of length $l$, there exists a factorization $z_3$ of length $q \neq l$ such that $d(z_1,z_3)<c_{eq}(m)$ and $d(z_2,z_3)<c_{eq}(m)$. \end{itemize}
\end{corollary}

\begin{proof}
    Since $c_{eq}(M)>c_{adj}(M)$, $c_{mon}(M)=c_{eq}(M)$. If an element $m \in M$ has $c_{eq}(m) \leq c_{adj}(m)$, then $c_{mon}(m)=c_{adj}(m)<c_{eq}(M)$ so $c(m)<c_{mon}(M)$. If $c_{eq}(m) > c_{adj}(m)$, then by the previous theorem, $c_{mon}(M) \geq c_{mon}(m)>c(m)$. So for all elements $m \in M$, $c(m)<c_{mon}(M)$. We conclude that $c_{mon}(M)>c(M)$. 
\end{proof}

Through our research, it has become evident that equivalent catenary degree plays a crucial role in determining monotone catenary degree. As demonstrated in this section, unless we have an element with multiple factorizations but only one factorization length, we can construct catenary graphs without ever connecting factorizations of the same length. Moreover, regular catenary degree almost never relies on equivalent catenary degree, whereas monotone catenary degree must be at least $c_{eq}(M)$. So when the equivalent catenary degree is small, we expect that $c(M)=c_{mon}(M)$. Likewise, when we have a large value for $c_{eq}(M)$, we expect $c(M)<c_{mon}(M)$. The following section further illustrates the importance of equivalent catenary degree; we see that equivalent catenary degree is quite easily understood in embedding dimension three. This allows us to find cases where the difference between monotone and regular catenary degrees grows arbitrarily large. 
\section{Embedding Dimension Three}

As seen previously, $c_{mon}(M)=\max\{c_{eq}(M),c_{adj}(M)\}$. By characterizing $c_{eq}(M)$, we can gain a deeper order understanding of the monotone catenary degree of these monoids. In many cases, the monotone catenary degree relies entirely on equivalent catenary degree, while $c_{eq}$ has no impact on regular catenary degree. In these cases, we can use generators that will yield large equivalent catenary degrees, thus producing large differences between regular and monotone catenary degrees. Since equivalent catenary degree behaves most nicely in embedding dimension three, we will focus entirely on embedding dimension three for the remainder of this section. Consider a monoid $M$ in embedding dimension three, $M=\langle n_1, n_2, n_3 \rangle$, with $n_1<n_2<n_3$.

\begin{lemma}
    In embedding dimension three, the only way to move between a factorization $(b,c,d)$ and a factorization of lower length is $\left(b-\frac{ln_1-k(n_3-n_1)}{n_2-n_1}-k-l,c+\frac{ln_1-k(n_3-n_1)}{n_2-n_1},d+k\right)$, where $l$ is the difference in length and $k$ is any integer. 
\end{lemma}
\begin{proof}
    Consider a factorization $(b,c,d)$ of some element $m \in M$. Now consider another factorization $(b_1,c_1,d_1)$ of $m$ of lower length, say $|(b,c,d)|=|(b_1,c_1,d_1)|+l$. Then $b+c+d=b_1+c_1+d_1+l$. Since these factorizations both factor the same element, we have
    \begin{align*}
        bn_1+cn_2+dn_3 &= b_1n_1+c_1n_2+d_1n_3 \\
        (b+c+d)n_1+c(n_2-n_1)+d(n_3-n_1) &= (b_1+c_1+d_1+l)n_1+c_1(n_2-n_1)+d_1(n_3-n_1)-ln_1 \\
        c(n_2-n_1)+d(n_3-n_1) &= c_1(n_2-n_1)+d_1(n_3-n_1)-ln_1 \\
    \end{align*}
    Suppose $d_1=d+k$ for some integer value $k$. Then 
    \begin{align*}
        c(n_2-n_1)+d(n_3-n_1) &= c_1(n_2-n_1)+(d+k)(n_3-n_1)-ln_! \\
        c(n_2-n_1) &= c_1(n_2-n_1)+k(n_3-n_1)-ln_1 \\
        c_1(n_2-n_1) &= c(n_2-n_1)-k(n_3-n_1)+ln_1 \\
        c_1 &= c + \left(\frac{ln_1-k(n_3-n_1)}{n_2-n_1}\right) \\
        b_1 &= b -  \left(\frac{ln_1-k(n_3-n_1)}{n_2-n_1}\right) -k-l \\ 
    \end{align*}
    Then $(b_1,c_1,d_1)= \left(b-\frac{ln_1-k(n_3-n_1)}{n_2-n_1}-k-l,c+\frac{ln_1-k(n_3-n_1)}{n_2-n_1},d+k\right)$. 
\end{proof}

\begin{corollary}
    In embedding dimension three, the only way to move between a factorization $(b,c,d)$ and a factorization of equal length is $\left(b + \frac{k(n_3-n_1)}{n_2-n_1} -k, c - \frac{k(n_3-n_1)}{n_2-n_1}, d+k\right)$.
\end{corollary}
\begin{proof}
    From the previous lemma, we have that the only way to move from $(b,c,d)$ to another factorization is $\left(b-\frac{ln_1-k(n_3-n_1)}{n_2-n_1}-k-l,c+\frac{ln_1-k(n_3-n_1)}{n_2-n_1},d+k\right)$. Here, $l$ is the difference in length between factorizations. Notice that when $l=0$, we have factorizations of the same length and we can write our new factorization as $\left(b + \frac{k(n_3-n_1)}{n_2-n_1} -k, c - \frac{k(n_3-n_1)}{n_2-n_1}, d+k\right)$. Then this is the only one way to move between factorizations of the same length. 
\end{proof}
\noindent The following example serves to illustrate this result. \newline
\begin{example}
    Let $M=\langle 4,9,19 \rangle$. Consider the factorizations of $105$ in $M$: 
    \begin{itemize}
        \item (24,1,0) has length 25
        \item (15,5,0) and (17,2,1) have length 20
        \item (6,9,0), (8,6,1), (10,3,2), and (12,0,3) have length 15
        \item (1,7,2), (3,4,3), and (5,1,4) have length 10
    \end{itemize}
    If we let $k=1$, we have $(b_1,c_1,d_1)=(b+\frac{15}{5}-1,c-\frac{15}{5},d+1)=(b+2,c-3,d+1)$. Notice that all factorizations of a given length can be obtained by applying this move.
\end{example}
\begin{lemma}
    In embedding dimension three, $c_{eq}(M) =   \frac{(n_3-n_1)}{gcd(n_3-n_1,n_2-n_1)}$.
\end{lemma} 
\begin{proof}
    We know that our only means of moving around factorizations of length $|(b,c,d)|$ is $\left(b + \frac{k(n_3-n_1)}{n_2-n_1} -k, c - \frac{k(n_3-n_1)}{n_2-n_1}, d+k\right)$. Consider the string of factorizations
    \begin{enumerate}
        \item[(1)] $(b,c,d)$
        \item[(2)] $\left(b + \frac{k(n_3-n_1)}{n_2-n_1} -k, c - \frac{k(n_3-n_1)}{n_2-n_1}, d+k\right)$
        \item[(3)] $\left(b + \frac{2k(n_3-n_1)}{n_2-n_1} -2k, c - \frac{2k(n_3-n_1)}{n_2-n_1}, d+2k\right)$ \newline
            \vdots \; \vdots \; \vdots \newline
        \item [(x)] $\left(b + \frac{xk(n_3-n_1)}{n_2-n_1} -xk, c - \frac{xk(n_3-n_1)}{n_2-n_1}, d+xk\right)$
    \end{enumerate}
    Note that these will only be factorizations when $\frac{k(n_3-n_1)}{n_2-n_1} \in \mathbb{Z}$. If we solve for the smallest $k$ for which $\frac{k(n_3-n_1)}{n_2-n_1}$ is an integer, we obtain $k= \frac{n_2-n_1}{\gcd(n_3-n_1,n_2-n_1)}$. Notice that this value of $k$ will allow us to step between all factorizations by $$\left(b + \tfrac{(n_3-n_1)}{gcd(n_3-n_1,n_2-n_1)} -\tfrac{n_2-n_1}{gcd(n_3-n_1,n_2-n_1)}, c - \tfrac{(n_3-n_1)}{gcd(n_3-n_1,n_2-n_1)}, d+\tfrac{n_2-n_1}{gcd(n_3-n_1,n_2-n_1)}\right).$$ The $\gcd$ of these factorizations is $\left(b, c -  \frac{(n_3-n_1)}{gcd(n_3-n_1,n_2-n_1)}, d\right)$ and the distance between factorizations is $ \frac{(n_3-n_1)}{gcd(n_3-n_1,n_2-n_1)}$. Hence, the smallest $N$ such that any two factorizations of the same length are connected by $N$-chains is  $ \frac{(n_3-n_1)}{gcd(n_3-n_1,n_2-n_1)}$, and we conclude that \newline $c_{eq}(M) =   \frac{(n_3-n_1)}{gcd(n_3-n_1,n_2-n_1)}$.
\end{proof}

\begin{theorem}
    In embedding dimension three, equivalent catenary degree is eventually constant with dissonance number $n_2\left(\frac{(n_3-n_1)}{gcd\left(n_3-n_1,n_2-n_1\right)}\right)+\mathcal{F}(M)$. Moreover, we can determine the equivalent catenary degrees of all elements $s$ of a monoid $M$:
    \begin{itemize}
        \item If $m = n_2\left(\frac{(n_3-n_1)}{gcd\left(n_3-n_1,n_2-n_1\right)}\right)+x$ where $x \in M$, then $c_{eq}(m)=c_{eq}(M)$. 
        \item If $m = n_2\left(\frac{(n_3-n_1)}{gcd\left(n_3-n_1,n_2-n_1\right)}\right)+x$ where $x \notin M$, then $c_{eq}(m)=0$. Note that in this case $x$ can be negative. 
    \end{itemize}
\end{theorem}
\begin{proof}
    For any monoid $M$ in embedding dimension three, we have that $c_{eq}(M) =  \frac{(n_3-n_1)}{gcd(n_3-n_1,n_2-n_1)}$. We also know that there is only one way to step between factorizations of the same length, which yields a difference of $\frac{(n_3-n_1)}{gcd(n_3-n_1,n_2-n_1)}$. So if an element $s$ has two factorizations of the same length, $c_{eq}(s)= \frac{(n_3-n_1)}{gcd(n_3-n_1,n_2-n_1)} = c_{eq}(M)$. If we can guarantee that an element $m$ has at least $\frac{(n_3-n_1)}{gcd(n_3-n_1,n_2-n_1)}$ copies of $n_2$, we guarantee that $m$ has $c_{eq}(m) =  \frac{(n_3-n_1)}{gcd(n_3-n_1,n_2-n_1)}$. Consider the element $m = n_2\left(\frac{(n_3-n_1)}{gcd\left(n_3-n_1,n_2-n_1\right)}\right)+x$ where $x \in M$. Then 
    \begin{align*}
        m &= n_2\left(\frac{(n_3-n_1)}{gcd\left(n_3-n_1,n_2-n_1\right)}\right)+(b,c,d) \text{ with nonnegative integers } b,c,d \\
        &= n_2\left(\frac{(n_3-n_1)}{gcd\left(n_3-n_1,n_2-n_1\right)}\right)+bn_1+cn_2+dn_3 \\
        &= bn_1+\left(c+\frac{(n_3-n_1)}{gcd\left(n_3-n_1,n_2-n_1\right)}\right)n_2+dn_3 \\
        &= \left(b,c+\frac{(n_3-n_1)}{gcd\left(n_3-n_1,n_2-n_1\right)},d\right)
    \end{align*}
    Then $m$ has at least $\frac{(n_3-n_1)}{gcd(n_3-n_1,n_2-n_1)}$ copies of $n_2$, and $c_{eq}(m) =  \frac{(n_3-n_1)}{gcd(n_3-n_1,n_2-n_1)}$. Then for any $m \geq n_2\left(\frac{(n_3-n_1)}{gcd\left(n_3-n_1,n_2-n_1\right)}\right)+\mathcal{F}+1$, $c_{eq}(m)=c_{eq}(M)$ and we conclude that $c_{eq}(m)$ is eventually constant. Now suppose we can write $m$ as 
    \begin{align*}
        m &= n_2\left(\frac{(n_3-n_1)}{gcd\left(n_3-n_1,n_2-n_1\right)}\right)+x \text{ where } x \notin M \\
        x &= m-n_2\left(\frac{(n_3-n_1)}{gcd\left(n_3-n_1,n_2-n_1\right)}\right)\\
    \end{align*}
    Assume for the sake of contradiction that $c_{eq}(m)=c_{eq}(M)$. Then $m=(b,c,d)$ where $c \geq \frac{(n_3-n_1)}{gcd(n_3-n_1,n_2-n_1)}$.
    \begin{align*}
        x &= bn_1+cn_2+dn_3-n_2\left(\frac{(n_3-n_1)}{gcd\left(n_3-n_1,n_2-n_1\right)}\right)\\
        &= bn_1+\left(c-\frac{(n_3-n_1)}{gcd\left(n_3-n_1,n_2-n_1\right)}\right)n_2+dn_3 \\
        &= \left(b,c-\frac{(n_3-n_1)}{gcd\left(n_3-n_1,n_2-n_1\right)},d\right)
    \end{align*}
    This implies that $x \in M$, but we know that this is not the case. This contradiction implies that if $m=n_2\left(\frac{(n_3-n_1)}{gcd\left(n_3-n_1,n_2-n_1\right)}\right)+x$ with $x \notin M$, $c_{eq}(m)=0$. Moreover, the dissonance number for $c_{eq}(m)$ is $n_2\left(\frac{(n_3-n_1)}{gcd\left(n_3-n_1,n_2-n_1\right)}\right)+\mathcal{F}(M)$. 

\end{proof}
Hence, we have fully characterized $c_{eq}(M)$ for any monoid $M$ in embedding dimension three. Notice that to obtain a large value for $c_{eq}(M)$, we require $n_3-n_1$ and $n_2-n_1$ to be relatively prime so that $c_{eq}(M)=n_3-n_1$. This value for equivalent catenary degree often yields a monotone catenary degree that is a great deal larger than regular catenary degree. To observe this, consider the following example.

\begin{example}
    Let $M=\langle 4,9,18 \rangle$ and let $M'=\langle 4,9,19 \rangle$. Then $c(M)=9$ and $c(M')=7$. For $M$, $\gcd(18-4,9-4)=\gcd(14,5)=1$ so $c_{eq}(M)=18-4=14$. In $M'$, $\gcd(19-4,9-4)=\gcd(15,5)=5$ so $c_{eq}(M')=\frac{19-4}{5}=3$. $c_{mon}(M)=c_{eq}(M)=14$, and $c_{mon}(M')=c_{adj}(M')=7$. Hence, $c_{mon}(M)>c(M)$, but $c_{mon}(M')=c(M)$. 
\end{example}

\section{The Difference Between Monotone and Regular Catenary Degree}

We now seek to discover classes of monoids that meet our conditions for $c_{mon}(M)>c(M)$. The following class of monoids not only meets these conditions, but can produce arbitrarily large differences between monotone and regular catenary degrees. 
Consider a monoid $M$ in embedding dimension three, $M=\langle na,na+n,2na+nx+1 \rangle$ with $x \geq 2$. 

\begin{lemma}
    Let $M=\langle na,na+n,2na+nx+1 \rangle$ with $x \geq 2$. Then $c_{eq}(M)=na+nx+1$.
\end{lemma}
\begin{proof}
    We know that $c_{eq}(M)= \frac{(n_3-n_1)}{gcd(n_3-n_1,n_2-n_1)} =  \frac{(na+nx+1)}{gcd(na+nx+1,n)}$. Since $n(a+x)+1$ and $n$ are relatively prime, $c_{eq}(M)=\frac{na+nx+1}{1}=na+nx+1$. 
\end{proof} $$ $$

\begin{lemma}
    Let $M=\langle na,na+n,2na+nx+1 \rangle$ with $x \geq 2$. Then $c_{adj}(M)<na+nx+1$.
\end{lemma}
\begin{proof}
    Remember that any factorization $(b_1,c_1,d_1)$ can be written as $(b-\frac{ln_1-k(n_3-n_1)}{n_2-n_1}-k-l,c+\frac{ln_1-k(n_3-n_1)}{n_2-n_1},d+k)=(b-\frac{lna-k(na+nx+1)}{n}-k-l,c+\frac{lna-k(na+nx+1)}{n},d+k)$ in terms of some other factorization $(b,c,d)$. So
    assume for the sake of contradiction that the shortest distance between $z_1=(b,c,d)$ and an adjacent factorization of lower length is at least $na+nx+1$. Notice that when $k=0$ and $l=1$, $z_i=(b-a-1,c+a,d)$ and so our difference is $a+1<na+nx+1$. Then suppose $z_i$ does not exist, so we force $b \leq a$. Similarly, when $k=n$ and $l=1$ we have $z_j=(b+na-a+nx-n,c-na+a-nx-1,d+n)$ which yields a difference of $(n-1)a+nx+1<na+nx+1$, so suppose $z_j$ does not exist. This forces $c \leq (n-1)a+nx$. Since $b \leq a$ and $c \leq (n-1)a+nx$, $b+c \leq na+nx$, so we need $d>d_1$ in order to obtain a difference of at least $na+nx+1$. To have $d>d_1$, we need a negative value for $k$. However, when $k$ is negative, $b_k=b-la+ka+kx+\frac{k}{n}-k-l<b-a<0$, so no such factorizations will exist. Hence, our adjacent catenary degree must be smaller than $na+nx+1$ and we conclude $c_{eq}(M)>c_{adj}(M)$. 
\end{proof}

\begin{theorem}
    Let $M=\langle na,na+n,2na+nx+1 \rangle$ with $x \geq 2$. Then $c_{mon}(M)=na+nx+1$
\end{theorem}

\begin{proof}
    We know that $c_{mon}(M)=\max(c_{eq}(M),c_{adj}(M))$ and that $c_{eq}(M)=na+nx+1$ and $c_{adj}(M)<na+nx+1$. Then $c_{mon}(M)=\max(c_{eq}(M),c_{adj}(M))=na+nx+1$. 
\end{proof}

\begin{lemma}
    Let $M=\langle na,na+n,2na+nx+1 \rangle$ with $x \geq 2$. If an element $m \in M$  has two factorizations $z$ and $z_1$ of a given length $l$, there exists another factorization $z_2$ of length $l-n$ such that $d(z,z_2)=nx+1$ and $d(z_1,z_2)=na+n$. 
\end{lemma}

\begin{proof}
    Suppose $m$ has two factorizations $z=(b,c,d)$ and $z_1=(b_1,c_1,d_1)$ of length $l$. We will prove that $m$ has a third factorization $z_2=(b_2,c_2,d_2)$ of length $l-1$. If $|z|=|z_1|=l$, then $(b_1,c_1,d_1)=(b+na+nx,c-na-nx-1,d+1)$. Let $(b_2,c_2,d_2)=(b_1-na-n,c_1+na,d_1)=(b+nx-n,c-nx-1,d+1)$. Then $nx-n>0$ so $b_2>0$, and since $c_1=c-na-nx-1 \geq 0$, $c_2=c-nx-1>0$. Clearly $d_2=d+1>0$. Then $(b_2,c_2,d_2)$ is another factorization of $m$ with length $l-n$. Note that $gcd(z,z_2)=(b,c-nx-1,d)$ and $d(z,z_2)=nx+1$. Also, $gcd(z_1,z_2)=(b_1-na-n,c_1,d_1)$ and $d(z_1,z_2)=na+n$. 
\end{proof}

\begin{theorem}
    $c_{mon}(M)>c(M)$
\end{theorem}

\begin{proof}
    We know that $c_{eq}(M)>c_{adj}(M)$ and that every element with at least two factorizations has at least two elements in its length set. First consider an element $x \in M$ with a single factorization. Then $c(x)=0$. Now consider an element $m$ with more than one factorizations. Then  $\mathcal{L}(s)$ has at least two elements. If $m$ has no two factorizations of the same length, then $c_{eq}(m)=0<c(s) \leq c_{adj}(m)<c_{eq}(M)$. If $m$ has at least two factorizations of a given length, then $c_{eq}(m)=na+nx+1>c_{adj}(m)$ so $c_{mon}(m)=na+nx+1$. When computing the catenary degree of this element, we can connect adjacent factorizations with $d<na+nx+1$. If two factorizations have length $l$, we can connect them to a factorization of length $l-n$ with distance $na+n$ or $nx+1$. In this way, we can construct the entire catenary graph of $s$ with distances less than $na+nx+1$. Hence, $c(m)<c_{eq}(m)$ for all $m \in M$. Then $c(M)<c_{mon}(M)$.   
\end{proof}

We now demonstrate that $c_{mon}(M)-c(M)$ can grow arbitrarily large in this class of monoids. 
Consider a monoid $M$ in embedding dimension three, $M=\langle a,a+1,\mathcal{F}\langle a,a+1 \rangle \rangle$. Here, $n=1$ and  $x=\frac{a^2}{2}-\frac{3a}{2}-1$. 

\begin{lemma}
    Let $M=\langle a,a+1,\mathcal{F}\langle a,a+1 \rangle \rangle$. Then $c_{mon}(M)=a^2-2a-1$.
\end{lemma}
\begin{proof}
    Note that $\mathcal{F}\langle a,a+1 \rangle=a(a+1)-(2a+1)=a^2+a-2a-1=a^2-a-1$. Since $gcd(a^2-2a-1,1)=1$, $c_{eq}(M)=a^2-2a-1$. We have already shown that in $M= \langle na,na+n,2na+nx+1 \rangle$, $c_{adj}<c_{eq}$. Then if we let $n=1$ and $x=\frac{a^2}{2}-\frac{3a}{2}-1$, we have $M= \langle a,a+1,2a+2(\frac{a^2}{2}-\frac{3a}{2}-1)+1 \rangle = \langle a,a+1,a^2-a-1 \rangle=\langle a,a+1,\mathcal{F}\langle a,a+1 \rangle \rangle$. So $c_{adj}(M)<c_{eq}(M)=a^2-2a-1$ and $c_{mon}(M)=\max(c_{eq}(M),c_{adj}(M))=a^2-2a-1$.
\end{proof}

\begin{lemma}
    The Betti Elements of $M$ are $a^2$, $a^2-1$, and $2\mathcal{F}\langle a,a+1 \rangle $. 
\end{lemma}
\begin{proof}
    First consider the element $a^2 \in M$. Clearly $(a,0,0)$ is a factorization of $a^2$. Note also that $(a+1)+(a^2-a-1)=a^2$ so $(0,1,1)$ is another factorization. We need to show that $a$ is the minimal coefficient $c$ such that $c \cdot a$ can be supported by $a+1$ and $\mathcal{F}\langle a,a+1 \rangle$. First we will show that $(a-1)a$ cannot be supported by $a+1$ and $\mathcal{F}\langle a,a+1 \rangle$. Since $\langle a+1,a^2-a-1 \rangle$ is symmetric, it is sufficient to show that $\mathcal{F}\langle a+1, a^2-a-1 \rangle - (a-1)a \in \langle a+1,a^2-a-1 \rangle$. 
    \begin{align*}
        \mathcal{F}\langle a+1, a^2-a-1 \rangle &= (a+1)(a^2-a-1)-(a^2) \\
        &= a^3-a^2-a+a^2-a-1-a^2 \\
        &= a^3-a^2-2a-1 \\
        \mathcal{F}\langle a+1, a^2-a-1 \rangle - (a-1)a &= a^3-a^2-2a-1-(a-1)a \\
        &=  a^3-a^2-2a-1-a^2+a \\
        &= a^3-2a^2-a-1 \\
    \end{align*}
    We need to show that $a^3-2a^2-a-1 \in \langle a+1,a^2-a-1 \rangle$. Notice that 
    \begin{align*}
        (a-3)(a+1) &= a^2-2a-3 \\
        (a-2)(a^2-a-1) &= a^3-3a^2+a+2 \\
        (a-3)(a+1)+(a-2)(a^2-a-1) &= a^3-2a^2-a-1
    \end{align*}
    Hence,  $\mathcal{F}\langle a+1, a^2-a-1 \rangle - (a-1)a \in \langle a+1,a^2-a-1 \rangle$ so $(a-1)a \notin \langle a+1,a^2-a-1 \rangle$. To see that coefficients less than $a-1$ will not work either, note that $(a-1)(a+1)-(a^2-a-1)=a$. Then 
    \begin{align*}
        \mathcal{F}\langle a+1, a^2-a-1 \rangle - (a-1)a &= (a-3)(a+1)+(a-2)(a^2-a-1) \\
        \mathcal{F}\langle a+1, a^2-a-1 \rangle - (a-2)a &= (2a-4)(a+1)+(a-3)(a^2-a-1)
    \end{align*}
    We can repeat this $a-2$ times until we have 
    $$ \mathcal{F}\langle a+1, a^2-a-1 \rangle -a=(a-3+(a-2)(a-1))(a+1) $$
    We conclude that $a$ is minimal, so $a^2$ is a Betti Element. Similarly, to prove that $a^2-1$ is a Betti Element, consider its factorizations: $(0,a-1,0)$ and $(1,0,1)$. We want to show that $a-1$ is the smallest coefficient $c$ such that $c \cdot (a+1)$ can be supported by $a$ and $a^2-a-1$. Note that $\mathcal{F}\langle a,a^2-a-1 \rangle = a^3-a^2-a-(a^2-1)=a^3-2a^2-a+1$, and \newline  $\mathcal{F}\langle a,a^2-a-1 \rangle - (a-2)(a+1) = (a^3-2a^2-a+1)-(a^2-a-2)=a^3-3a^2+3$. 
    \begin{align*}
        \mathcal{F}\langle a,a^2-a-1 \rangle - (a-2)(a+1) &= a^3-3a^2+3 \\
        &= a^3-4a^2+2a+3+a^2-2a \\
        &= (a-3)(a^2-a-1)+(a-2)a \\
    \end{align*}
    Then $(a-2)(a+1) \notin \langle a,a^2-a-1 \rangle$. Since $a(a)-(a^2-a-1)=a+1$, we have 
    $$\mathcal{F}\langle a,a^2-a-1 \rangle - (a-3)(a+1) = (a-4)(a^2-a-1)+(2a-2)a$$
    We can repeat this process $a-3$ times until we have 
    $$\mathcal{F}\langle a,a^2-a-1 \rangle - (a+1) = ((a-3)a-2)a$$
    So $a-1$ is minimal and we conclude that $a^2-1$ is a Betti Element. Finally, we will demonstrate that $2\mathcal{F}\langle a,a+1 \rangle $ or $2a^2-2a-2$ is a Betti Element. Consider its factorizations: $(0,0,2)$ and $(a-1,a-2,0)$. To see that the latter is a factorization of $2a^2-2a-2$, 
    \begin{align*}
        (a-1)a+(a-2)(a+1) &= a^2-a+a^2-a-2 \\
        &= 2a^2-2a-2
    \end{align*}
    To see that $2$ is the minimal coefficient $c$ such that $c \cdot \mathcal{F}\langle a,a+1 \rangle \in \langle a,a+1 \rangle$, we just need to show that $1 \cdot \mathcal{F}\langle a,a+1 \rangle \notin \langle a,a+1 \rangle$. By definition, this is true since the Frobenius number is never in its monoid. Then $2$ is minimal and $2\mathcal{F}\langle a,a+1 \rangle $ is a Betti Element. 
\end{proof} 

\begin{lemma}
    Let $M=\langle a,a+1,\mathcal{F}\langle a,a+1 \rangle \rangle$. Then $c(M)=2a-3$.
\end{lemma}
\begin{proof}
    It is known that the catenary degree of a monoid occurs at a Betti Element. Since we have computed the Betti Elements of $M$, we can determine $c(M)$. Consider the Betti Element $a^2$ and its factorizations: $a^2 = (a,0,0) = (0,1,1)$.
    Note that $gcd = (0,0,0)$ and $d=a$. Then $c(a^2)=a$. Now consider the Betti Element $a^2-1$: $a^2-1 = (0,a-1,0) = (1,0,1)$. 
    Then $gcd=(0,0,0)$ and $d=a-1$. So $c(a^2-1)=a-1$. Finally, consider the Betti Element $2\mathcal{F}\langle a,a+1 \rangle $: $2\mathcal{F}=(0,0,2)=(a-1,a-2,0)$.
    Then $gcd=(0,0,0)$ and $d=2a-3$. So $c(2\mathcal{F}\langle a,a+1 \rangle)=2a-3$. Then $c(M)=\max(a,a-1,2a-3)=2a-3$. 
\end{proof}

\begin{theorem}
    The difference between $c_{mon}(M)$ and $c(M)$ can be arbitrarily large. 
\end{theorem}
\begin{proof}
    Let $M=\langle a,a+1,\mathcal{F}\langle a,a+1 \rangle \rangle$. We know that $c_{mon}(M)=a^2-2a-1$ and $c(M)=2a-3$. Then \newline $c_{mon}(M)-c(M)=a^2-4a-4$, which can grow arbitrarily large for large $a$. 
\end{proof}

The following graph illustrates this result. As we increase $a$ in the class of monoids $M=\langle a,a+1,\mathcal{F}\langle a,a+1 \rangle \rangle$, the difference between monotone and regular catenary degrees grows arbitrarily large. 

\begin{figure}
    \centering
    \includegraphics[width=10cm]{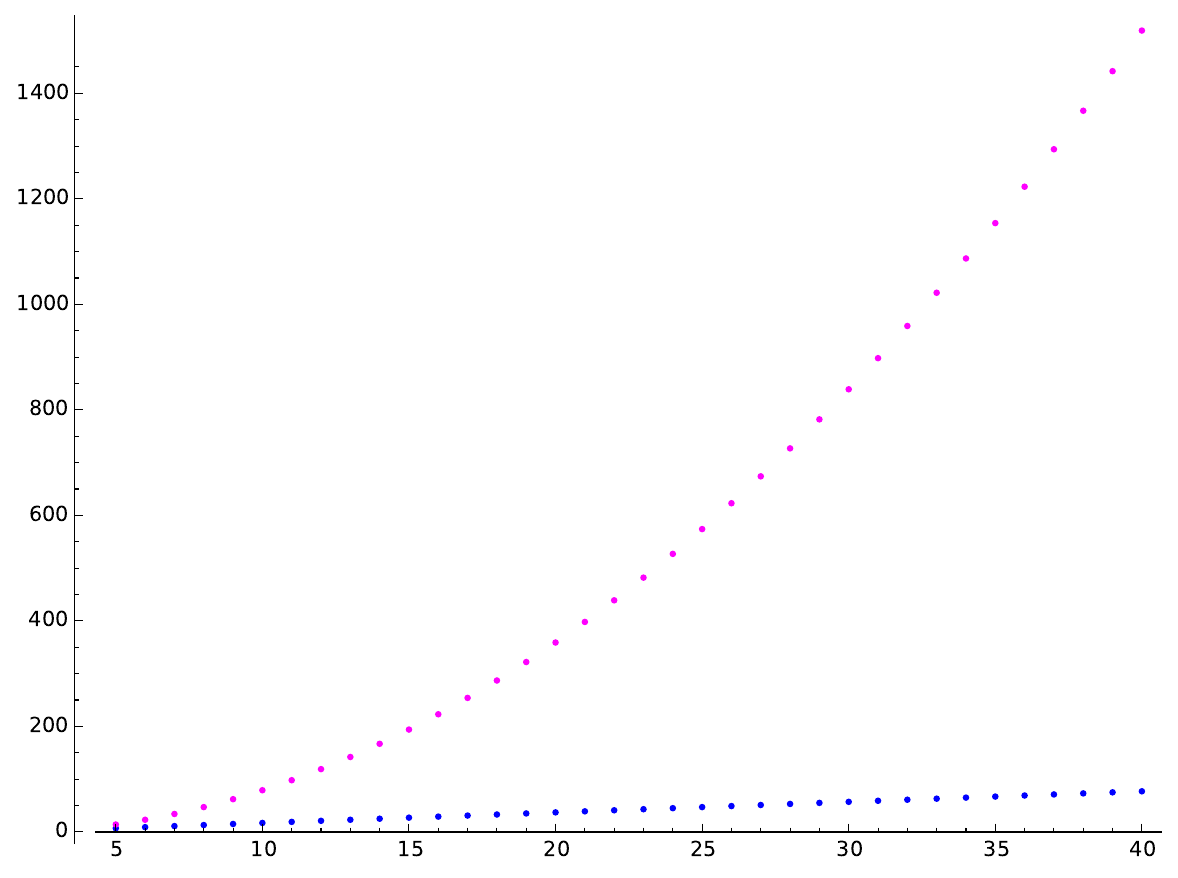}
    \caption{Monotone and Regular Catenary Degrees of $\langle a,a+1,\mathcal{F}\langle a,a+1\rangle \rangle$}
    \label{fig:Arbitrarily Large}
\end{figure}

\end{document}